\documentclass[12pt,a4paper,twoside]{article}
\usepackage{amsmath,amsfonts,amsthm,amsopn,color,amssymb,enumitem}
\usepackage{palatino}
\usepackage{graphicx}

\newcommand{\eps}{\varepsilon}

\newcommand{\R}{\mathbb{R}}

\newcommand{\RN}{{\mathbb{R}^N}}

\renewcommand{\le}{\leqslant}
\renewcommand{\ge}{\geqslant}
\renewcommand{\a }{\alpha }

\renewcommand{\b }{\beta }

\renewcommand{\d }{\delta }

\newcommand{\g }{\gamma }

\renewcommand{\l }{\lambda}
\newcommand{\n }{\nabla }

\def\bbm[#1]{\mbox{\boldmath $#1$}}
\newcommand{\beq }{\begin{equation}}
\newcommand{\eeq }{\end{equation}}

\newtheorem{theorem}{Theorem}[section]
\newtheorem{lemma}[theorem]{Lemma}

\newtheorem{remark}[theorem]{Remark}

\renewenvironment{proof}{\noindent{\textbf{Proof\quad}}}{$\hfill\square$\vspace{0.2 cm}\\}

\title{{\bf
Ground state solution for a problem \\with mean curvature operator
in Minkowski space. \footnote{The author is supported by M.I.U.R. -
P.R.I.N. ``Metodi variazionali e topologici nello studio di fenomeni
non lineari''}}}

\author{A. Azzollini \thanks{Dipartimento di Matematica, Informatica ed Economia, Universit\`a degli
Studi della Basilicata,  Via dell'Ateneo Lucano 10, I-85100 Potenza,
Italy, e-mail: {\tt antonio.azzollini@unibas.it}}}

\date{}

\begin{document}
\maketitle

\begin{abstract}
In this paper we prove the existence of a radial ground state
solution for a quasilinear problem involving the mean curvature
operator in Minkowski space.
\end{abstract}

\section*{Introduction}
In this paper we study the following quasilinear problem
    \begin{equation}
\left\{
\begin{array}{ll}
\n \cdot \left[\frac{\n u}{\sqrt{1-|\n u|^2}}\right] + f(u) = 0, &
x\in \RN, \label{eq}
\\
u(x)>0,\quad \hbox{in }\R^N\\ u(x) \to 0 , \quad \hbox{as }|x|\to
\infty,
\end{array}
\right.
\end{equation}
where $N\ge 2$ and $f:\R\to\R$.\\
The differential operator we are considering, known as the mean
curvature operator in the Minkowski space, has been deeply studied
in the recent years, in nonlinear equations on bounded domains with
various type of boundary conditions (see \cite{BJM, BJM2, BJT, BJT2}
and the references within) and in the whole $\R^N$ for
nonlinearities $f$ of
the type $u^p$ (see \cite{BDD}).\\
If we look for radial solutions, we can reduce equation \eqref{eq}
to the following ODE
    \begin{equation}\label{eq:ode}
        \left(\frac{u'}{\sqrt{1-(u')^2}}\right)'+\frac{N-1}{r}\frac{u'}{\sqrt{1-(u')^2}}+f(u)=0
    \end{equation}
where $u\in C^2([0,+\infty])$ is such that $u'(0)=0.$\\
We will use the shooting method to establish the global existence of
the solutions of the Cauchy problem
    \begin{equation}\label{cauchy}
\left\{
\begin{array}{ll}
\left(\frac{u'}{\sqrt{1-(u')^2}}\right)'+\frac{N-1}{r}\frac{u'}{\sqrt{1-(u')^2}}+f(u)=0
\\
u(0)=\xi, u'(0) =0
\end{array}
\right.
\end{equation}
where $\xi$ is allowed to vary in an interval which we will define
later. As usual, in this type of problem the local existence is not
difficult to prove,
since standard fixed point theorems work fine.\\
What is really interesting is to find the conditions which permit to
extend the solution to the whole $\R_+$ and to prove that the
solution is a ground state,
namely $\lim_{r\to\infty}u(r)=0$.\\
The shooting argument has been used in the past to find ground state
solutions to various types of equations. We recall two significant
examples such as
    \begin{equation}\label{BLP}
        \Delta u+f(u)=0,
    \end{equation}
treated in \cite{BLP} or the following prescribed mean curvature
equation
    \begin{equation}\label{PS}
        \n\cdot\left(\frac{\n u}{\sqrt{1+|\n u|^2}}\right)-\l u+u^q
        =0,
    \end{equation}
studied in \cite{PS}. The method consists in studying the profile of
the solution of \eqref{cauchy} as the initial value $\xi$ varies
into an interval. In particular, since we are interested in ground
states, we aim to exclude the cases in which for a finite $R>0$
either $u$ or $u'$ vanishes. Using the property of the intervals to
be connected, if we proved that the values $\xi$ corresponding to
the {\it bad} cases constitute two open disjoint non empty subsets
of an interval $I$, we should have found at least an initial value
whose corresponding
solution is a ground state.\\
We make the following assumptions over $f$
    \begin{enumerate}
      \item[({\bf f1})] $f(0)=0$,
      \item[({\bf f2})] $f$ is locally Lipschitz in $[0,+\infty)$,
      \item[({\bf f3})] $\exists \a:=\inf\{\xi>0\mid f(\xi)\ge 0\}>0$,
      \item[({\bf f4})] (if $N\ge3$) $\lim_{s\to\a^+}\frac{f(s)}{s-\a}>0$,
      \item[({\bf f5})] $\exists \g>0$ such that $F(\g):=\int_0^\g f(s)\,ds >0$,
    \end{enumerate}
      and, defining
        \begin{equation}\label{eq:xiz}
            \xi_0:=\inf\{\xi>0\mid F(\xi)>0\},
        \end{equation}
      we assume
    \begin{enumerate}
      \item[({\bf f6})] $f(\xi)>0$ in $(\a,\xi_0].$
    \end{enumerate}

In the sequel, we will suppose that $f$ is extended in $\R_-$ by
$0$. Of course, since we are looking for positive solutions, this
assumption does not involve the generality of the problem. The main
result of the paper is the following
    \begin{theorem}\label{main}
        If
        \begin{itemize}
            \item $N\ge 3$ and $f$ satisfies
            ({\bf f1}-$\ldots$-{\bf f6})
            \item $N=2$ and $f$ satisfies ({\bf f1}), ({\bf f2}), ({\bf f3}), ({\bf f5}) and
            ({\bf f6}),
        \end{itemize}
        then \eqref{eq} has a radially decreasing solution.
    \end{theorem}

    \begin{remark}
        We do not treat the case $N=1$ since it is definitely
        analogous to $u''+f(u)=0$. Then we refer to \cite[Section 6]{BL} for
        sufficient and necessary condition for the existence of the
        unique solution of the problem
            \begin{equation*}
                \left\{
                    \begin{array}{ll}
                        \left(\frac{u'}{\sqrt{1-(u')^2}}\right)'+f(u)=0
                        \\
                        u(x)>0,\quad \hbox{in }\R
                        \\ u(x) \to 0 , \quad \hbox{as }|x|\to
                        \infty.
                    \end{array}
                \right.
            \end{equation*}
    \end{remark}

    \begin{remark}
        We exhibit some examples of functions $f$ satisfying our
        assumptions
            \begin{enumerate}
                \item $f(s)=-\l s + s^q$ for $\l>0$ and $q>1$
                is a nice function for
                $N\ge 2,$
                \item $f(x)=-s\sin(s)|\sin(s)|^{q-1}$ is nice
                when $N\ge 2$ and $q=1$ and when $N=2$ and $q\ge 1$.
            \end{enumerate}
    \end{remark}
        \begin{remark}
        By comparing our main result with those in \cite{BLP} and \cite{PS}, some remarkable
        differences stand out. For example we point out that no assumption is required on
        the behaviour of $f$ at infinity. On the contrary, when for instance $f$ is as in example 1,
        a necessary condition both in \cite{BLP} and
        in \cite{PS} is $q\in(1,\frac{N+2}{N-2})$,
        for $N\ge 3$.\\
        Moreover the existence result proved in \cite{PS} holds for $\l$ sufficiently small.
        On the other hand a nonexistence result
        has been proved for \eqref{PS} in \cite{FLS} when
        $\l>\left(2\frac{q+1}{q-1}\right)^{\frac{q-1}{q+1}}.$ As
        shown in example 1, in our case $\l$ is allowed to be any positive number.
    \end{remark}

\section{Proof of the existence result}

Observe that the solution of \eqref{cauchy} satisfies the equation
    \begin{equation}\label{eq:phi}
        (r^{N-1}\phi'(u'))'=-r^{N-1}f(u),
    \end{equation}
where $\phi(s):=1-\sqrt{1-s^2}$ (for $s\in[-1,1]$).\\
It is easy to verify that $\phi':]-1,1[\to\R$ is an increasing
diffeomorphism. Set $\d>0$ (whose smallness will be later
established) and denote by $C:=C(\R_+,\R)$ and by
$C_\d:=C([0,\d],\R)$ respectively the set of the continuous
functions defined in $\R_+$ and in the interval $[0,\d].$ Define the
following operators
    \begin{equation*}
        S:C\to C, \quad S(u)(r):=
            \left\{
                \begin{array}{ll}
                    -\frac 1{r^{N-1}} \int_0^r t^{N-1}
                    u(t)\,dt&\hbox{if } r>0,\\
                    0&\hbox{if } r=0,\\
                \end{array}
            \right .
    \end{equation*}
and
    \begin{equation*}
        K:C\to C, \quad K(u)(r)=\int_0^r u(t)\,dt.
    \end{equation*}

For every $\xi\in \R$, define the translation operator $T_\xi:C\to
C$ such that $T_\xi(u)=\xi+u.$ Moreover, consider the Nemytskii
operators associated to $f$ and $(\phi')^{-1}$,
    \begin{align*}
        &N_f:C\to C,\quad N_f(u)(r)=f(u(r)),
        \\
        &N_{(\phi')^{-1}}:C\to C,\quad
        N_{(\phi')^{-1}}(u)(r)=(\phi')^{-1}(u(r)).
    \end{align*}
Set $\rho >0$ and denote with $B_\rho:=\{u\in C_\d\mid
\|u\|_{\infty}\le\rho\}.$ We set the following fixed point problem:
for any $\xi\in\R$ we want to find $u\in \xi+B_\rho$ such that
    \begin{equation}\label{eq:fixed}
        u=T_\xi\circ K\circ N_{(\phi')^{-1}}\circ S\circ N_f(u).
    \end{equation}
Since $(\phi')^{-1}$ and $f$ are respectively Lipschitz and locally
Lipschitz, Banach-Caccioppoli fixed point theorem guarantees the
existence of a sufficiently small $\d>0$ such that the function
$u:=u(\xi,r)\in \xi+B_\rho$ is a solution of \eqref{eq:fixed}. It is
easy to observe that $u$ is a local solution of the Cauchy problem
\eqref{cauchy}.

    Now, let $R>0$ be such that $[0,R)$ is the maximal interval
    where the function $u$ is defined.
    Multiplying \eqref{eq:ode} by
    $u'$ and integrating over $(0,r)$ we obtain the following
    equality for any $r\in (0,R)$
        \begin{equation}\label{eq:int}
            H(u'(r))+(N-1)\int_0^r\frac{(u')^2(s)}{s\sqrt{1-(u')^2(s)}}\,ds
            = F(\xi)-F(u(r))
        \end{equation}
    where
    $H(t)=\frac{1-\sqrt{1-t^2}}{\sqrt{1-t^2}}.$\\
    Denote by
        \begin{equation}\label{eq:beta}
            \beta:=\inf\{\xi>\xi_0\mid f(\xi)=0\}.
        \end{equation}
    Of course $\a<\xi_0<\beta\le+\infty.$ Denote by $I$ the interval $(\a,\beta)$
    and take $\xi\in I.$ By ({\bf f3}) and ({\bf f6}), for every $s\le\beta$ we
have $F(s)\ge F(\a),$ so from \eqref{eq:int} we  deduce that
$H(u'(r))$ is bounded as far as $u(r)\le \beta$.

Observe that, since $f(u(0))=f(\xi)>0$, from equation \eqref{eq:ode}
we deduce that $u''(0)<0$ and then there exists $\eta>0$ such that
$u'(r)<0$ and $\xi>u(r)>0$ for every $r\in (0,\eta).$ Set
    \begin{equation}\label{eq:barr}
        \bar R:=\left\{
\begin{array}{ll}
\inf\{r\in (0,R)\mid u'(r)\ge 0\}&\hbox{if } u'(r)=0 \hbox{ for some
} r\in (0,R)
\\
+\infty&\hbox{otherwise}.
\end{array}
\right.
    \end{equation}
    \begin{remark}\label{rem}
        According to the definition
        \eqref{eq:barr} we have that $0<\eta\le\bar R\le +\infty$ and, since $u(r)<\xi<\b$ for
        every $r\in (0,\bar R),$ from \eqref{eq:int} we have
            \begin{equation}\label{eq:bound}
                \exists \eps>0 \hbox{ such that, for any } r\in (0,\bar R), |u'(r)|\le 1-\eps.
            \end{equation}
        In particular we deduce that $\bar R=+\infty$ implies $R=+\infty.$
    \end{remark}

Define the following two intervals
    \begin{align*}
        I_+:&=\left\{\xi\in I \Big | \exists R'\le R \hbox{ such that}
            \begin{array}{l}
            u(\xi,r)>0, u'(\xi,r)<0, \hbox{for } r<R'\\
            u'(\xi,R')=0
            \end{array}
        \right\}\\
        \noalign{and}
        I_-:&=\left\{\xi\in I \Big | \exists R'\le R \hbox{ such that}
            \begin{array}{l}
            u(\xi,r)>0, u'(\xi,r)<0, \hbox{for } r<R'\\
            u(\xi,R')=0
            \end{array}
        \right\}.
    \end{align*}
We will prove that $I_+$ and $I_-$ do not cover $I$.
    \begin{lemma}\label{le:lim}
        Suppose $R=+\infty$ and $u'(r)<0,$ $u(r)>0$ for every $r>0$. Then $\lim_{r\to+\infty}
        u(r)=0.$
    \end{lemma}
    \begin{proof}
        Of course by monotonicity there exists
        $l=\lim_{r\to+\infty}u(r)\ge 0.$ By \eqref{eq:ode} and
        \eqref{eq:bound}, we deduce that
            \begin{equation}\label{eq:exlim}
                \lim_{r\to+\infty}\left(\frac{u'(r)}{\sqrt{1-(u'(r))^2}}\right)'=-f(l).
            \end{equation}
            Suppose that $f(l)\neq 0,$ say $f(l)>0.$ By simple computations, from \eqref{eq:bound} and \eqref{eq:exlim}
            we deduce that, definitively, $u''(r)<-\d<0,$ for some $\d>0.$
            Of course this is not possible because of \eqref{eq:bound}.\\
            Since $f(l)=0,$ there are only two possibilities, either $l=0$ or
            $l=\a.$

            Suppose $N=2$ and, by contradiction, $l=\a.$ Since for any $r>0$
            $\b>u(r)>\a,$ from \eqref{eq:phi} we deduce that
            $r\phi'(u'(r))$ is decreasing in $\R_+$ and then, in particular, there
            exists $R_0>0$ and $\d>0$ such that for any $r>R_0$ we
            have $\phi'(u'(r))<-\frac \d r.$ By \eqref{eq:bound} we
            infer that, for some $M>0$, we have $M u'(r)\le
            \phi'(u'(r))$ and then
                \begin{equation*}
                    u'(r)\le-\frac \d {Mr}\quad\hbox{for any }r>R_0.
                \end{equation*}
            Integrating in $(R_0,r)$ we obtain
                \begin{equation*}
                    u(r)\le u(R_0)- \frac \d M \log \left(\frac
                    {r}{R_0}\right)
                \end{equation*}
            which contradicts $l=\a.$

            Suppose $N\ge 3.$ To prove $l\neq \a$, assume by contradiction that $l=\a$. Computing in
            \eqref{eq:ode}, we have that the following equality
            holds in $(0,+\infty)$:
            \begin{equation*}
                \frac{u''}{[1-(u')^2]^{\frac 32}}=-\frac{N-1}{r}\frac{u'}{\sqrt{1-(u')^2}}-
                    f(u).
            \end{equation*}
            Taking into account \eqref{eq:bound}, there exists $\d>0$ such that
            $\d\le \sqrt{1-(u')^2}\le 1$. We deduce that
            \begin{equation}\label{eq:secdev}
                u''=-\frac{N-1}{r}u'[1-(u')^2]-  f(u)[1-(u')^2]^{\frac 32}\le-\frac{N-1}{r}u'- \d^3 f(u)
            \end{equation}
            where we have used the fact that $u'<0$ and $f(u)>0.$ Now we proceed as in \cite{BLP},
            repeating the arguments for completeness. If we
            set
            $v=r^{\frac
            {N-1}2}u$, by \eqref{eq:secdev} we get the following estimate
                \begin{equation}\label{eq:est}
                    v''\le
                    \left[\frac{(N-1)(N-3)}{r^2}-\d^3\frac{f(u)}{u}\right]v
                \end{equation}
            from which, in view of assumption ({\bf f4}), we deduce that $v''$ is definitively
            negative. Now, since $v'$ is definitively decreasing,
            certainly there exists
            $L=\lim_{r\to+\infty}v'(r)<+\infty.$\\
            Of course $L$ can
            not be negative, since otherwise
            $\lim_{r\to+\infty}v(r)=-\infty$.\\
            On the other hand, if $L\ge0,$ then we
            deduce that $v$ is definitively increasing and then there
            exists $R_0>0$ such that for any $r>R_0$ we have $v(r)>
            v(R_0).$ From \eqref{eq:est} we infer that, for some positive constant $C$, $v''(r)\le
            -C<0$ definitively and this implies
            $L=\lim_{r\to+\infty}v'(r)=-\infty$: again a
            contradiction.
    \end{proof}

    \begin{theorem}\label{th:nonemp1}
        $I_+$ is not empty.
    \end{theorem}
    \begin{proof}
        Let $\xi\in(\a,\xi_0)$. By \eqref{eq:xiz}, $F(\xi)<0.$ By \eqref{eq:int} we deduce that
        $F(u(r))< F(\xi)<0$ for any $r\in (0,R).$ As a consequence, by ({\bf f6}) we have that there exists $m >0$ such that
            \begin{equation}\label{eq:pos}
                0<m < u(r)< \xi,
            \end{equation}
        and then, by Remark \ref{rem}, $R=+\infty.$ Now, assuming that $u'(r)<0$ for any $r>0$,
        by Lemma \ref{le:lim} we get a contradiction with
        \eqref{eq:pos}.
    \end{proof}

    Now, to prove that $I_-$ is not empty, we need some preliminary
    results. Consider the problem
            \begin{equation}
\left\{
\begin{array}{ll}
\n \cdot \left[\frac{\n u}{\sqrt{1-|\n u|^2}}\right] + f(u) = 0, &
\hbox{in } B_\rho, \label{eqb}
\\
u = 0 , & \hbox{on } \partial B_\rho.
\end{array}
\right.
\end{equation}
If $\b<+\infty$ (we recall that $\beta$ is defined in
\eqref{eq:beta}), we replace $f$ in \eqref{eqb} by
    \begin{equation}\label{eq:repl}
        \tilde f(s)=\left\{
            \begin{array}{ll}
                f(s)&
                \hbox{if }s\le\b,\\
                0&\hbox{if }s>\b.
            \end{array}
        \right.
    \end{equation}
As in \cite{BJT}, we use a variational approach to \eqref{eqb}.

Set $W_\rho:= W^{1,\infty}((0,\rho),\R).$ It is well known that
$W_\rho\hookrightarrow C_\rho.$

Define
    \begin{equation*}
        K_0:=\{u\in W_\rho\mid \|u'\|_\infty\le 1, u(\rho)=0\}
    \end{equation*}
and
    \begin{equation*}
        \Psi(u):=\left\{
            \begin{array}{ll}
                \displaystyle\int_0^\rho r^{N-1}(1-\sqrt{1-(u')^2})\,dr&
                \hbox{if }u\in K_0\\
                +\infty &\hbox{if }u\in W_\rho\setminus K_0.
            \end{array}
        \right.
    \end{equation*}
For any $u\in W_\rho$ we set
    \begin{equation*}
        J(u):=\Psi(u)-\int_0^\rho r^{N-1}
        F(u)\,dr.
    \end{equation*}
It is easy to verify that the functional $J$ is a Szulkin's
functional (see \cite{Sz}) so that, by \cite[Proposition 1.1]{Sz},
we have that if $u\in W_\rho$ is a local minimum of $J$, then it is
a Szulkin critical point and  for any $v\in K_0$ it solves the
inequality

    \begin{equation}\label{eq:Sz}
        \int_0^\rho r^{N-1}(\phi(v')-\phi(u'))\,dr-\int_0^\rho r^{N-1}
        f(u)(v-u)\,dr\ge 0
    \end{equation}
where we recall that $\phi$ is defined in \eqref{eq:phi}.
    \begin{lemma}\label{le:min}
        If $u_0\in K_0$ is a local minimum for $J$, then $u_0(|x|)$ is a classical solution of
        \eqref{eqb}.
    \end{lemma}
    \begin{proof}
        We will use an argument taken from \cite{BM}.

        Suppose $u_0\in K_0$ is a minimum for $J$ and consider the problem
            \begin{equation}\label{eq:aux}
                (r^{N-1}\phi'(v'))'-r^{N-1}v= -r^{N-1}(f(u_0)+u_0),\quad
                v'(0)=0, v(\rho)=0.
            \end{equation}
        By \cite[Theorem 2.1]{BJM}, certainly \eqref{eq:aux} has a
        classical
        solution. As in \cite[Lemma 3, Lemma 4]{BM} we deduce that the solution is
        unique, call it $\bar v$, and for any $w\in K_0$ it satisfies the following inequality
            \begin{equation}\label{eq:anineq}
                        \int_0^\rho r^{N-1}(\phi(w')-\phi(\bar v'))\,dr+\int_0^\rho r^{N-1}
        (\bar v-f(u_0)-u_0)(w-\bar v)\,dr\ge 0.
            \end{equation}
        Now write \eqref{eq:Sz} for $v=\bar v$ and \eqref{eq:anineq}
        for $w=u_0$ and sum up the two inequalities. What we obtain is
            \begin{equation*}
                -\int_0^{\rho}r^{N-1}(u_0-\bar v)^2\,dr\ge0
            \end{equation*}
        which implies $u_0=\bar v$ and then $u_0$ is the unique classical solution of
        \eqref{eq:aux}. We conclude that $u_0(|x|)$ is a classical solution of
        \eqref{eqb}.
    \end{proof}

    \begin{theorem}\label{th:nonemp2}
        $I_-$ is not empty.
    \end{theorem}
    \begin{proof}
        As a first step, we show that
            \begin{itemize}
                \item[1.] $J$ is bounded below and achieves its
                infimum,
                \item[2.] if $\rho>0$ is sufficiently large, then $c_0=\inf_{u\in W_\rho} J(u)<0$.
            \end{itemize}
        Observe that
            \begin{equation*}
                \forall u\in K_0:\quad \|u\|_{\infty}\le \rho.
            \end{equation*}
        As a consequence, it is easy to see that $J$ is bounded below.
        Consider $(u_n)_n\in W_\rho$ a minimizing sequence. Of course we can
        assume $u_n\in K_0$ for any $n\ge 1.$
        By Ascoli Arzel\`a theorem, there exists a subsequence, relabeled $(u_n)_n$, and a continuous function $u_0$ such that
            \begin{equation}\label{eq:convunif}
                u_n\to u_0\quad\hbox{uniformly in } [0,\rho].
            \end{equation}
        To prove that $u_0$ is in $K_0,$ we just observe that, for any $x,y\in [0,\rho]$, with $x\neq y$,  we have
            \begin{equation*}
                \lim_n\frac{u_n(x)-u_n(y)}{x-y}=\frac{u_0(x)-u_0(y)}{x-y},
            \end{equation*}
        and then also $u_0$ has Lipschitz constant 1.\\
        By \eqref{eq:convunif} and \cite[Lemma 1]{BM}, $\Psi(u_0)\le\liminf_n\Psi(u_n).$ Then, again by \eqref{eq:convunif}, we have
            \begin{equation*}
                J(u_0)\le c_0.
            \end{equation*}
        Now we prove our second claim. Consider the following
        function defined for $\rho>2\g$
            \begin{equation*}
                    w_\rho(r)=\left\{
            \begin{array}{ll}
                \g&
                \hbox{in }[0,\rho-2\g]\\
                \frac{\rho-r}2 &\hbox{in }[\rho-2\g,\rho].
            \end{array}
        \right.
    \end{equation*}
    Of course $w_\rho\in K_0$. Moreover
        \begin{align*}
            J(w_\rho) &\le \frac 12 \int_{\rho-2\g}^\rho(2-\sqrt 3) s^{N-1}\, ds\\
            &\qquad-F(\g) \frac {(\rho-2\g)^N}{N}+\frac 1N \max_{0\le s\le \g} |F(s)|(\rho^N-(\rho-2\g)^N)\\
            &\le C_1(\rho^N-(\rho-2\g)^N)-\frac {F(\g)}{N} (\rho-2\g)^N\\
            &\le C_2\rho^{N-1}-C_3\rho^N
        \end{align*}
    where $C_1,C_2$ and $C_3$ are suitable positive constant. The second claim is an obvious consequence of the previous chain of inequalities.

    Now, suppose $\rho_0>0$ and $u_0\in K_0$ are such that $I(u_0)=c_0<0$ and set $\bar\xi=u_0(0).$
    The value $\bar\xi\in (\a,\b).$ Indeed, by Lemma \ref{le:min}, $u_0(|\cdot|)$ is a classical
    solution of \eqref{eqb} and then $u_0$ is a local solution of \eqref{cauchy}, with $\xi=\bar\xi$
    and $\tilde f$ instead of $f$ if $\beta<+\infty.$ If $\bar\xi\le\a,$ then $F(\bar\xi)\le 0$ leads to an obvious
    contradiction to \eqref{eq:int} computed in $r=\rho_0.$ On the other hand, $\bar\xi$ can not
    be greater than $\b$, since in this case, by \eqref{eq:repl}, the unique solution of the Cauchy
    problem \eqref{cauchy} would be the constant function $u(r)=\bar\xi.$\\
    By contradiction, suppose that $\bar\xi\notin I_-.$ Since we can assume $u_0(r)>0$ in $[0,\rho_0)$,
    otherwise we consider the function $u_0$ restricted to the interval $[0,R')$ where $R':=\inf\{r>0\mid u_0(r)=0\}$,
    our contradiction assumption implies that $\bar R\in (0,\rho_0)$ (the definition of $\bar R$ is given in
    \eqref{eq:barr}).\\
    Computing \eqref{eq:int} for $r=\bar R$ and for $r=\rho_0$, we respectively have
        \begin{align}
            &(N-1)\int_0^{\bar R}\frac{(u')^2(s)}{s\sqrt{1-(u')^2(s)}}\,ds
            = F(\bar\xi)-F(u(\bar R)),\label{eq:int1}\\
            &H(u'(\rho_0))+(N-1)\int_0^{\rho_0}\frac{(u')^2(s)}{s\sqrt{1-(u')^2(s)}}\,ds
            = F(\bar\xi).\label{eq:int2}
        \end{align}
    Subtracting \eqref{eq:int1} from \eqref{eq:int2}, we obtain
        \begin{equation*}
            H(u'(\rho_0))+(N-1)\int_{\bar R}^{\rho_0}\frac{(u')^2(s)}{s\sqrt{1-(u')^2(s)}}\,ds
            = F(u(\bar R))
        \end{equation*}
    that is $F(u(\bar R))>0.$\\ Since $u'(r)<0$ for any $r\in(0,\bar R)$, we have that
    $u''(\bar R)\ge 0$ and then from \eqref{eq:ode} it follows that $f(u(\bar R))\le 0.$
    Since $f$ is positive in $I$ and $0<u(\bar R)<\bar\xi<\beta,$ certainly $u(\bar R)\in (0,\a].$
    From this we deduce that $F(u(\bar R))<0$ and then the contradiction.
    \end{proof}

    \begin{theorem}\label{th:disop}
        $I_+$ and $I_-$ are disjoint and open.
    \end{theorem}
    \begin{proof}
        By contradiction, suppose $\bar \xi\in I_+\cap I_-.$ Then,
        since the solution of \eqref{cauchy} with $\xi=\bar \xi$ is
        such that $u(R')=u'(R')=0,$ we can extend it by $0$ in $(R',+\infty)$ and we get a compact support
        solution to the equation \eqref{eq:ode}.  Simple computations
        shows that this contradicts the strong maximum principle as it
        appears in \cite[Theorem 1]{PSZ} (actually this theorem concerns a different class of operators,
        but the proof works also in our case),
        since $u(|x|)$ would be a compact support solution to the
        equation
        in
        \eqref{eq}.\\
        An alternative (and simpler) proof consists in observing
        that, by uniqueness theorem, $u=0$ is the unique solution of
        the Cauchy problem
            \begin{equation*}
                \left\{
                    \begin{array}{ll}
                        \left(\frac{u'}{\sqrt{1-(u')^2}}\right)'+\frac{N-1}{r}\frac{u'}{\sqrt{1-(u')^2}}+f(u)=0
                        \\
                        u(R')=0, u'(R') =0.
                    \end{array}
                \right.
            \end{equation*}
        Finally, observe that, by continuous dependence on the initial datum, $I_+$
        and $I_-$ are open sets.
    \end{proof}

    By Theorem \ref{th:nonemp1}, \ref{th:nonemp2} and
    \ref{th:disop}, we can take $\xi\in I\setminus(I_+\cup I_-).$ Since
    $\bar R=+\infty,$ by Remark \ref{rem} $u(\xi,r)$ is defined in $\R_+.$
    By Lemma \ref{le:lim} $\lim_{r\to+\infty}u(\xi,r)=0.$ As a
    consequence $\bar u(x)=u(\xi,|x|)$ is a solution of \eqref{eq}.


\begin{thebibliography}{99}

    \bibitem{BL}
H. Berestycki, P.L. Lions, {\it Nonlinear scalar field equations. I.
Existence of a ground state}, Arch. Rational Mech. Anal. {\bf 82}
(1983), 313--345.

\bibitem{BLP}
H. Berestycki, P.L. Lions, L.A. Peletier
    {\it An ODE approach to the existence of positive solutions for semilinear problems in
    $\RN$}, Indiana Univ. Math. J. {\bf 30} (1981), 141–-157.


\bibitem{BJM}
C. Bereanu, P. Jebelean, J. Mawhin,
    {\it Radial solutions for some nonlinear problems involving
    mean curvature operators in Euclidean and Minkowski spaces}, Proc.
Amer. Math. Soc., {\bf 137} (2009), 171--178.

\bibitem{BJM2}
C. Bereanu, P. Jebelean, J. Mawhin,
    {\it Radial solutions for Neumann problems involving mean curvature
operators in Euclidean and Minkowski spaces}, Math. Nachr. {\bf 283}
(2010), 379-–391.

\bibitem{BJT}
C. Bereanu, P. Jebelean, P.J. Torres, {\it Positive radial solution
for Dirichlet problems with mean curvature operators in Minkowski
space}, J. Functional Analysis {\bf 264} (2013), 270--287.

\bibitem{BJT2}
C. Bereanu, P. Jebelean, P.J. Torres, {\it Multiple positive radial
solutions for a Dirichlet problem involving the mean curvature
operator in Minkowski space}, J. Functional Analysis {\bf 265}
(2013), 644--659.

\bibitem{BDD}
D. Bonheure, A. Derlet, C. De Coster,  {\it Infinitely many radial
solutions of a mean curvature equation in Lorentz-Minkowski space},
Rend. Istit. Mat. Univ. Trieste {\bf 44} (2012), 259–-284.

\bibitem{BM}
    H. Brezis, J. Mawhin,
{\it Periodic solution of the forced relativistic pendulum},
Differential Integral Equations {\bf 23} (2010), 801--810.

\bibitem{FLS}
B. Franchi, E. Lanconelli, J. Serrin,
    {\it Esistenza e unicit\'a degli stati fondamentali per equazioni ellittiche
    quasilineari},
    Rendiconti Acc. Naz. dei Lincei {\bf 79} (1985), 121--126.

\bibitem{PS}
L.A. Peletier, J. Serrin, {\it Ground states for the prescribed mean
curvature equation}, Proc. Amer. Math. Soc. {\bf 100} (1987),
694--700.

\bibitem{PSZ}
P. Pucci, J. Serrin, H. Zou, {\it A strong maximum principle and a
compact support principle for singular elliptic inequalities}, J.
Math. Pures Appl. {\bf 78} (1999), 769--789.

\bibitem{Sz}
A. Szulkin, {\it Minimax principles for lower semicontinuous
functions and applications to nonlinear boundary value problems},
Ann. Inst. H. Poincar\`e Anal. Non Lin\`eaire {\bf 3} (1986),
77--109.

\end{thebibliography}
\end{document}